\documentclass{amsart}
\usepackage{amsmath, hyperref}

\usepackage{xy}
\input xy
\xyoption{all}

\title[Subvarieties of the  hyperelliptic moduli]{Subvarieties of the  hyperelliptic moduli determined by group actions}
\author{T. Shaska}
\address{Department of Mathematics and Statistics, Oakland University, Rochester, MI, 48309-4485, USA.}
\email{shaska@oakland.edu}


\newtheorem{thm}{Theorem}
\newtheorem{lemma}[thm]{Lemma}

\newtheorem{remark}[thm]{Remark}

\newtheorem{ex}{Example}

\def\Z{\mathbb Z}

\def\C{\mathbb C}
\def\bP{\mathbb P}

\def\L{{\mathcal H}_g (G, \sigma)}

\def\H{\mathcal H}
\def\M{\mathcal M}

\def\w{\widetilde}
\def\l{\lambda}
\def\s{\sigma}

\def\a{\alpha}

\def\b{\beta}

\def\P{\mathcal W}
\def\e{\xi_5}
\def\iso{\equiv}

\def\o{\otimes}
\def\w{\omega}
\def\g{\gamma}

\def\iso{{\, \cong\, }}

\def\<{\langle}
\def\>{\rangle}

\def\Aut{\mbox {Aut}}
\def\bAut{\overline {\mbox {Aut}}}
\def\G{\overline G}

\def\ff{\phi}
\def\f{\Phi}
\def\X{\mathcal X}

\def\d{\delta}

\newenvironment{mat2}{\left(\begin{array}{cc}}{\end{array}\right)}
\def\bC{{\bf C}}
\def\nf{\Psi}
\def\df{\Upsilon}

\def\div{\mathfrak d}
\def\ddiv{\bar \div}

\begin{document}

\begin{abstract}
Let $\H_g$ be the moduli space of genus $g$ hyperelliptic curves. In this note, we study the locus $\L$ in
$\H_g$ of curves admitting a $G$-action of given ramification type $\s$    and inclusions between such loci.
For each genus we determine the list of all possible groups, the inclusions among the loci, and the
corresponding equations of the generic curve in $\L$. The proof of the results is based solely on
representations of finite subgroups of $PGL_2 (\C)$ and the Riemann-Hurwitz formula.
\end{abstract}

\maketitle

\section{Introduction}
Let $\H_g$ be the moduli space of genus $g$ hyperelliptic curves. We study the locus $\L$ in $\H_g$ of curves
admitting a $G$-action of given ramification type $\s$. All components of $\L$ have the same dimension which
depends only on the signature of the $G$-action. Restricting the action to a subgroup $H$ of $G$ yields an
inclusion of $\L$ into the corresponding locus $\H_g (H, \s)$ for the action of $H$. In this paper we study
all possible loci $\L$, the inclusions between such loci, and determine an equation for a generic curve $C$
in $\L$.   This is the first part of the two paper sequence, the second of which will consider such problem
for hyperelliptic curves defined over a field of characteristic $p > 0$. The main goal of this paper is
twofold: first to give a unified treatment of automorphisms groups of hyperelliptic curves and the loci they
determine in characteristic zero,  and second to provide the motivation needed for studying such problem in
characteristic $p > 0$. While some of the results of this paper are scattered in the literature we provide a
unified approach which is algebraic and is based solely on representations of finite subgroups of $PGL_2
(\C)$ and the Riemann-Hurwitz formula.

In section 2, we give a brief introduction on hyperelliptic curves and  their automorphism groups. This
material can be found in \cite{BS, Bu, issac, jaa} among many other places in the literature.   For a given
hyperelliptic curve $\X$, defined over $k$,  with automorphism group $G$, the reduced automorphism group  is
$\G:=G/\< w\>$, where $w$ is the hyperelliptic involution. This group $\G$ is embedded in $PGL_2(k)$ and
therefore is one of $\Z_n D_n, A_4, S_4, A_5$. $\G$ acts on e genus 0 field $k(x)$. We determine a rational
function $\phi(x)$ that generates  the fixed field $k(x)^{\G}$ in all cases. Using only this rational
function we are able to determine the parametric equation of each family $\H_g (G, \s)$ (cf. Section~4).
Different decompositions of $\phi (x)$ give different decompositions of $f(x)$ in the equations of the
hyperelliptic curve $y^2=f(x)$.  An equation of an hyperelliptic curve with an extra automorphism of order
$n$ can be written as $y^2=f(x^n)$ or $y^2=x f(x^n)$.  This corresponds to a decomposition of the rational
function $\phi (x)$ in $x^n$. For algorithms on decomposing rational functions one can check \cite{G} and the
references from there. Furthermore, for each fixed $g$ we give a formula for the number of automorphism
groups that occur.

In section 3, we discuss the locus $\H_g (G, \s)$. First, we determine the signatures of the covers $\psi:
\X_g \to \bP_1$. The moduli space of such covers with fixed group $G$ and ramification $\s$ is a Hurwitz
space $\H$. There is an obvious map from $\H$ to the hyperelliptic moduli of curves $\H_g$. We denote the
image of this map  by $\H_g (G, \s)$ which is a subvariety of $\H_g$. The dimension of $\H_g (G, \s)$ is
determined solely by the signature $\s$. For each $g \geq2 $ we list all possible groups, their signatures,
and the dimension of the locus $\H_g (G, \s)$. In section 4, we determine the equations of the families of
curves for a given group. This is determined using the rational function $\phi (x)$ from section 2.
Weierstrass points of the curve are points in the fibers $\phi^{-1} (\l)$, where $\l$ is a branch point of
$\phi (x)$. Such branch points can be determined easily when $\phi (x)$ is known. For each group an equation
for the family of curves is determined.

In section 5, we discuss the inclusions among the loci $\H_g (G, \s)$. We implement a program that for each
genus $g\geq 2$ determines the list of groups which occur as full automorphism group of hyperelliptic curves
of genus $g$ and draws a the lattice of these groups.

There is plenty of literature on the automorphism groups of hyperelliptic curves. Among many papers we
mention \cite{BS},  \cite{GS}. Most of these papers have studied  the automorphism groups of the
hyperelliptic curve using the Fuchsian groups. Our goal is to provide a unified simple algebraic approach of
the results with the list of all the groups $G$ which occur as full automorphism groups of hyperelliptic
curves, all possible signatures $\s$ for each given group $G$, the dimension of each locus $\L$, and the
lattice of the loci $\L$.

\medskip

\noindent \textbf{Notation:} Throughout this paper $k$ denotes an algebraically closed field of
characteristic zero, $g$ an integer $\geq 2$, and $\X_g$ a hyperelliptic curve of genus $g$ defined over $k$.

\section{Hyperelliptic curves and their automorphisms}
Let $k$  be an algebraically  closed field of characteristic  zero and $\X_g$  be a genus  $g$ hyperelliptic
curve given  by the equation $y^2=F(x)$, where $\deg(F)=2g+2$. Denote  the function field of $\X_g$ by
$K:=k(x,y)$. Then, $k(x)$ is the  unique degree 2 genus zero subfield of  $K$. $K$ is  a quadratic extension
field of $k(x)$ ramified exactly at $d=2g+2$ places $\a_1, \dots , \a_d$  of $k(x)$. The corresponding places
of $K$ are called the {\it Weierstrass points} of $K$. Let $\P:=\{ \a_1, \dots , \a_d \}$ and $G=Aut(K/k)$.
Since $k(x)$  is the only  genus 0 subfield of degree  2  of $K$, then  $G$  fixes $k(x)$.  Thus,
$G_0:=Gal(K/k(x))=\< z_0 \>$, with $z_0^2=1$, is central in $G$. We call  \emph{the reduced automorphism
group} of $K$ the  group $\G:=G/G_0$. Then, $\G$ is isomorphic to one of the following: $ \Z_n$, $D_n$,
$A_4$, $S_4$, $A_5$ with branching indices of the corresponding cover $\bP^1_x \to \bP^1/ \G$ given
respectively by $(n,n),\  (2, 2, n),\ (2, 3, 3),\ (2, 4, 4),\ (2, 3, 5).$

We fix a coordinate $z$ in $\bP^1/ \G$. Thus, $\G$ is the monodromy group of a cover $\ff: \bP^1_x \to
\bP^1_z$. We denote by $q_1, \dots , q_r $ the corresponding branch points of $\ff$. Let $S$ be the set of
branch points of $\f: \X_g \to \bP^1_z$. Clearly $q_1, \dots , q_r \in S$. As above $W$ denotes the images in
$\bP^1$ of Weierstrass points of $\X_g$ and $V:=\cup_{i=1}^r \ff^{-1} (q_i) $. For each $q_1, \dots , q_r$ we
have a corresponding permutation $\s_1, \dots , \s_r \in S_n$. The tuple $\bar \s:=(\s_1, \dots , \s_r)$ is
the signature of $\G$. Thus, $\G= \< \s_1, \dots , \s_r \>$ and $\s_1 \cdots \s_r =1$.

Since each of the above groups is embedded in $PGL_2 (\C)$ then we can have these generating systems $\s_1,
\dots , \s_r$ as matrices in $PGL_2 (\C)$. Below we display all the cases:
\begin{equation}\label{eq1}
\begin{split}
i) \quad \Z_n  & \iso \Big{\<}\begin{mat2}   \zeta_n & 0 \\ 0 & 1 \\ \end{mat2},
\begin{mat2}   \zeta_n^{n-1} & 0 \\ 0 & 1 \\ \end{mat2} \Big{\>} \\
ii) \quad D_n & \iso \Big{\<}
\begin{mat2} 0 & 1\\1 & 0 \\\end{mat2}, \begin{mat2} 0 & 1\\1 & 0 \\\end{mat2},
\begin{mat2}   \zeta_n & 0 \\ 0 & 1 \\ \end{mat2}
\Big{\>} \\
iii) \quad A_4 & \iso \Big{\<}
\begin{mat2}  -1 & 0 \\ 0 & 1 \\ \end{mat2},
\begin{mat2} 1 & i  \\ 1 & -i  \\ \end{mat2}
\Big{\>}  \\
iv) \quad S_4 & \iso \Big{\<}
\begin{mat2}   -1 & 0 \\ 0 & 1 \\ \end{mat2},
\begin{mat2}   0 & -1 \\ 1 & 0 \\ \end{mat2},
\begin{mat2} -1 & -1 \\ 1 &  1 \\ \end{mat2}
 \Big{\>}  \\
v) \quad A_5 & \iso \Big{\<}
\begin{mat2}   \w & 1 \\ 1 & -\w \\ \end{mat2},
\begin{mat2}   \w & \e^4 \\ 1 & -\e^4\w \\ \end{mat2}
\Big{\>} \\
\end{split}
\end{equation}
where $\omega=\frac{-1+\sqrt{5}}{2}$,  $\zeta_n$ is a primitive $n^{th}$ root of unity, $\e$ is a primitive
$5^{th}$ root of unity, and $i$ is a primitive $4^{th}$ root of unity.

\subsection{Fixed fields of the reduced automorphism groups}
The group $\G$ given above acts on $k(x)$ via the natural way. The fixed field is a genus 0 field, say
$k(z)$. Thus, $z$ is a degree $| \G |$ rational function in $x$, say $z=\phi(x)$. In this section we
determine $\phi(x)$ and its decompositions.
\begin{lemma}\label{lem1}
Let $H$ be a finite subgroup of $PGL_2(k)$. Let us identify each element of $H$ with the corresponding
Moebius transformation and let $s_i$ be the $i$-th elementary symmetric polynomial in the elements of $H$,
$i=1,\ldots,|H|$. Then any non-constant $s_i$ generates $k(z)$.
\end{lemma}
\begin{proof}
It is easy to check that the $s_i$ are the coefficients of the minimum polynomial of $x$ over $k(z)$. It is
well-known that any non-constant coefficient of this polynomial generates the field.
\end{proof}

\begin{lemma}
The fixed field for each of the groups $\G$ in cases i) - v) is generated respectively by the function
\begin{itemize}
\item[i)] $z= x^n $

\item [ii)] $z= x^n + \frac 1 {x^n}$

\item [iii)] $z=\frac {x^{12} -33x^8 -33 x^4+1} {x^2 (x^4-1)^2}$

\item [iv)] $z= \frac {(x^8 + 14 x^4 +1 )^3} {108 \left( x (x^4-1) \right)^4}$

\item [v)] $z= \frac {\left(-x^{20}+228 x^{15}-494x^{10}-228x^5-1 \right)^3}
{1728 \left(x(x^{10} + 11x^5 -1)\right)^5}$
\end{itemize}
\end{lemma}

\begin{proof}
Apply  Lemma~\ref{lem1} to the embedding of $\G$ given above.
\end{proof}

Notice that the branch points of a rational function $\phi (x)= \frac {f(x)} {g(x)}$ are exactly the zeroes
of the discriminant of the polynomial $r(x):=f(x)-t\cdot g(x)$ with respect to $x$. Then the branch points of
each of the above functions are
i) $\{0, \infty\}$,
ii) $\{ -2, 2, \infty\}$,
iii) $\{\infty, -6i\sqrt{3}, 6i\sqrt{3}\}$,
iv) $\{0, 1, \infty\}$,
v)  $\{ 0, 1728, \infty\}$.
The above facts are well known in the literature, see for example Klein \cite{Kl}.

\subsection{Decomposition of $\phi(x)$}
The automorphism group of $k(x)/k(\phi)$ is the embedding of $\G$ detailed before. As $|\G|=[k(x):k(\phi)]$,
there is a degree-preserving correspondence between subgroups of $\G$ and intermediate fields in the
extension. By L\"uroth's Theorem, each of those fields is $k(h)$ for some rational function $h$. Now, it is
clear that, in general, $k(f)\subset k(h) \Leftrightarrow f=g\circ h\mathrm{\ for\ some\ } g$. Thus, we can
use computer algebra techniques to find all the decompositions of $\phi$ and describe the lattice of
intermediate fields.

It is clear from the expression of $\phi$ that there is a decomposition $\phi=g(x^s)$ for $s$ taking the
values $n, n, 2,4,5$ respectively.  This comes also from the fact that the subgroup $\< \e \cdot x\>$ of $\G$
corresponds to the field generated by $x\cdot \e x\cdot \cdots \e^{s-1} x=x^s.$

Finding different decompositions of $\phi (x)$ is not a trivial computational problem. There are algorithms
available to do this; see \cite{G} for details.

For example, for $\G = A_5$ it is also possible to find decompositions involving $x^2$ or $x^3$ for functions
that are equivalent to $\phi$. Namely, for any $\sigma\in PGL_2(k)$, a generator of the field fixed for the
conjugate group $\sigma A_5\sigma^{-1}$ is $\phi(\sigma^{-1})$. If $\sigma$ is chosen in such a way as having
$\{x,-x\}<\sigma A_5 \sigma^{-1}$, then $k(x\cdot(-x))=k(x^2)$ will be an intermediate field by Lemma 1. This
can be accomplished by conjugating any involution of $A_5$ into $-x$. In the same manner, if an element of
order 3 in $A_5$ is conjugated into $\zeta_3 x$, where $\zeta_3$ is a primitive cubic root of $1$, the
resulting function can be written in terms of $x\cdot\zeta_3 x\cdot\zeta_3^2 x=x^3$; see \cite{a5} for
details.

\subsection{Automorphism groups and their signatures}
The automorphism groups of hyperelliptic curves have been classified by \cite{BS}, \cite{Bu}. Most of these
results study automorphism groups in terms of the Fuchsian groups. Since, we take the algebraic approach we
go over some of the results briefly.

The automorphism group $G$ of the hyperelliptic curve is a degree 2 central extension of $\G$.  The following
lemma is proved in \cite{GS}.
\begin{lemma}
Let $p\geq 2$, $\a \in G$ and $\bar \a$ its image in $\G$ with order $| \, {\bar \a} \, |= p $. Then,

i) $| \, \a \, | = p$ if and only if it fixes no Weierstrass points.

ii) $| \, \a \, | = 2p$ if and only if it fixes some Weierstrass point.
\end{lemma}

Let $W$ denote the images in $\bP^1_x$ of Weierstrass places of $\X_g$ and $V:=\cup_{i=1}^3 \phi^{-1} (q_i)$.

Let $z= \frac {\nf (x) } { \df (x) }$, where $\nf, \df \in k[x]$. For each branch point $q_i$, $i=1,2,3$ we
have the degree $|\G|$ equation $z\cdot \df(x) - q_i \cdot \df (x) =\nf (x),$ where the multiplicity of the
roots correspond to the ramification index for each $q_i$ (i.e., the index of the normalizer in $\G$ of
$\s_i$). We denote the ramification of $\phi: \bP^1_x \to \bP^1_z$, by $\varphi_m^r, \chi_n^s,  \psi_p^t$,
where the subscript denotes the degree of the polynomial.

Let $\l \in S \setminus \{ q_1, q_2, q_3\}$. The points in the fiber of a non-branch point $\l$ are the roots
of the equation: $ \nf (x) -\l \cdot \df (x) =0.$ To determine  the equation of the curve we simply need to
determine the Weierstrass points of the curve. For each fixed $\phi$ there are the following eight cases:

\begin{equation}\label{cases}
\begin{split}
1) & \quad  \, V \cap W = \emptyset,\\
2) & \quad  \, V \cap W =\phi^{-1} (q_1),\\
3) & \quad  \, V \cap W =\phi^{-1} (q_2),\\
4) & \quad  \, V \cap W =\phi^{-1} (q_3),\\
5) & \quad  \, V \cap W =\phi^{-1} (q_1) \cup \phi^{-1} (q_2),\\
6) & \quad  \,V \cap W =\phi^{-1} (q_2) \cup \phi^{-1} (q_3),\\
7) & \quad  \, V \cap W =\phi^{-1} (q_1) \cup \phi^{-1} (q_3),\\
8) & \quad  \, V \cap W =\phi^{-1} (q_1) \cup \phi^{-1} (q_2) \cup \phi^{-1} (q_3).
\end{split}
\end{equation}

\noindent It turns out that the above cases also determine the   full automorphism groups. We define the
following groups  as follows:
\begin{equation}\label{groups1}
\begin{split}
V_n := & \<  \, \, x, y \, | \, x^4, y^n,  (xy)^2, (x^{-1}y)^2\, \>, \quad
H_n :=  \<\,  x, y \, \, | \, \, x^4, y^2x^2, (xy)^n \, \>,   \\
G_n := & \< \, x, y \, \, | \, \, x^2 y^n, y^{2n},  x^{-1} y x y \, \>,\quad
U_n :=  \< \, x, y \, | \, x^2, y^n,  xyxy^{n+1} \>,  \\
\end{split}
\end{equation}
Sometimes these groups are called \textbf{twisted dihedral, double dihedral, generalized quaternion}, and
\textbf{semidihedral}. We warn the reader that these terms are not standard in the literature. They are all
four degree 2 central extensions of the dihedral group $D_n$ and therefore have order $4n$. Notice that $V_2$
is isomorphic with the dihedral group of order 8 and $H_2 \iso U_2 \iso \Z_2 \o \Z_4$. Furthermore, we have
the following result, the proof is elementary and we skip the details.

\begin{remark}  i) If $n\equiv 1 \mod 2$ then $H_{4n} \iso G_{4n}$

 ii) If $n=2^{s+1}$ then $G_n = Q_{2^{s+1}}$ for any $s\in \Z$.
\end{remark}

\noindent Further, the following groups
\begin{equation*}
\begin{split}
W_2:=  & \< \, x, y \, | \, x^4, y^3, y x^2  y^{-1} x^2, (x y)^4 \>, \quad
W_3:=   \< \, x, y \, | \, x^2, y^3, x^2 (x y)^4, (x y)^8 \> \\
\end{split}
\end{equation*}
are degree 2 central extensions of $S_4$. Now we have the following result.
\begin{thm}
The full automorphism group of a hyperelliptic curve is isomorphic to one of the following
$\Z_2 \o \Z_n$,  $\Z_n$,
$\Z_2 \o D_n$, $V_n$, $D_n$, $H_n$, $G_n$, $U_n$,
$\Z_2 \o A_4$, $SL_2(3)$, $\Z_2 \o S_4$,  $GL_2 (3)$, $W_2$, $W_3$
$\Z_2 \o A_5$, $SL_2 (5)$.
Furthermore, the signature for each group is as in Table~\ref{tab_1}.
\end{thm}

\begin{proof}
The ramification of $\phi : \bP^1_x \to \bP^1/ \G$   is one of the following $(n,n)$, $(2, 2, n)$, $(2, 3,
3)$,  $(2, 4, 4)$, $(2, 3, 5)$.  Recall that $|G|=2 | \G |$. Each case of Eq.~\eqref{cases} determines a
group of automorphisms.

Let $\G=\Z_n$ and $\G=< \a>$. If $\a$ fixes no Weierstrass points then, from the above Lemma, $\a$ lifts to
an element of order $n$ in $G$. Hence, $G=\Z_2 \o \Z_n$. If $\a$ fixes some Weierstrass points then $\a$
lifts to an element of order $2n$ in $G$ and $|G|=2n$ then $G$ is the cyclic group of order $2n$.

The cases left have three branch points for the cover  $\phi : \bP^1_x \to \bP^1/ \G$. From the above lemma
we have that if the places in the fiber $\phi^{-1} (q_1)$, $\phi^{-1} (q_2)$, $\phi^{-1} (q_3)$, are
Weierstrass points then $\s_1, \s_2, \s_3$ lift in $G$ to elements of orders $2 |\s_1|$, $2 |\s_2|$, and $2
|\s_3|$ respectively.

Let $\G\iso D_n$ where $D_n$ is given as in Eq.~\eqref{eq1}. Since the branching of $q_1$ and $q_2$ is the
same then there are basically six distinct cases which could arise. In other words, cases 2 and 3 from
Eq.~\eqref{cases} give the same group $G$. The same happens for cases 6 and 7 from Eq.~\eqref{cases}.

If none of the places in the fibers $\phi^{-1} (q_1)$, $\phi^{-1} (q_2)$, $\phi^{-1} (q_3)$, are Weierstrass
points then $\s_1, \s_2, \s_3$ lift in $G$ to elements of orders $|\s_1|$, $ |\s_2|$, and $ |\s_3|$
respectively. Together, with the hyperelliptic involution we have $G= \Z_2 \o D_n$. When places in $\phi^{-1}
(q_1)$ (or $\phi^{-1} (q_2)$) are Weierstrass points then the involution $\s_1 \in D_n$, as in
Eq~\eqref{eq1},  lifts to an element of order 4. In this case, the group has generators
\[ G= \<  \, \, \bar \s_1 , \bar \s_3 \, | \, {\bar \s_1}^4, {\bar \s_3}^n,  (\bar \s_1 \bar \s_3)^2,
({\bar \s_1}^{(-1)}{\bar \s_3})^2\, \>,\]
where $\bar \s_1, \bar \s_3$ are the lifts of $\s_1, \s_3$ in $G$. Thus, $G\iso V_n$.

When places in $\phi^{-1} (q_3)$  are Weierstrass points then the element of order $\s_1 \in D_n$ of order
$n$ lifts to an element of order $2n$. In this case, the group $G$ has generators of order $2n$ and the
hyperelliptic involution. Thus, $G \iso D_{2n}$.

The other three cases from Eq.~\eqref{cases}, namely case 5), 6) or 7), and case 8) give groups $H_n$, $U_n$,
$G_n$ with presentation as in Eq.~\eqref{groups1}.

If $\G \iso A_4$ then again we have two branch cycles which are the same. Hence, we have 6 cases. When the
involution lifts to an element of order 4 then the degree 2 central extension of $A_4$ is $SL_2 (3)$,
otherwise the extension is $\Z_2 \o A_4$.

The proof when $\G \iso S_4$ or $A_5$ goes similarly and we skip the details. The signature for each group
follows accordingly for each case.

\end{proof}

The above theorem was first proven in \cite{Bu} using Fuchsian groups. Notice that representations of groups
are given in that paper.

\begin{tiny}
\begin{table}
\begin{center}
\renewcommand{\arraystretch}{1.24}
\begin{tabular}{||c|c|c|c|c|c|c ||}
 \hline \hline \# & $G$ & $\G$ & $\d (G, \bC)$ & $\d, \, n, \,g  $ & $\bC=(C_1, \dots C_r)$ & $ \ff$ \\
 \hline \hline
 1 & $\Z_2 \o \Z_n $ & &$\frac {2g+2} n -1$&$n < g+1$ & $(n^2, n^2, 2^n, \dots , 2^n)$ &  \\
 2 &  $\Z_{2n}$       &$\Z_n$  &$\frac {2g+1}n -1$  &  & $(n^2, 2n, 2^n, \dots , 2^n )$& $(n, n)$ \\
 3 &   $\Z_{2n}$       & &$\frac {2g} n -1 $ &$n < g$ & $(2n, 2n, 2^n, \dots , 2^n)$  &   \\
   \hline \hline
4 & $\Z_2 \o D_n$  &     &$\frac {g+1} n$         &     & $( n^4, 2^{2n}, \dots , 2^{2n} )$ &  \\
5 & $V_n$          &     &$\frac {g+1}n-\frac 1 2$&   & $(n^4, 4^n, 2^{2n}, \dots , 2^{2n} )$&   \\
6 & $D_{2n}$       &$D_n$&$\frac g n$             &                 & $((2n)^2, 2^{2n}, \dots , 2^{2n} )$& $(2^n, 2^n, n^2)$    \\
7 & $H_n$          &     &$\frac {g+1} n -1$      &$ n < g+1$  &$(4^n, 4^{n}, n^4, 2^{2n} \dots , 2^{2n} )$ &    \\
8 & $U_n$          &     &$\frac g n- \frac 1 2$  &$ g \neq 2$  &$(4^n, (2n)^2, 2^{2n}, \dots , 2^{2n})$ &     \\
9 & $G_n$          &     &$\frac g n - 1$         &$n < g$   &$(4^n, 4^n, (2n)^2, 2^{2n}, \dots , 2^{2n})$    &       \\
\hline \hline
10 & $\Z_2\o A_4$ &  &$\frac {g+1} 6$&  &$( 3^8, 3^8, 2^{12}, \dots , 2^{12} )$ &    \\
11 & $\Z_2\o A_4$ & &$\frac {g-1} 6$&  & $( 3^8, 6^4, 2^{12}, \dots , 2^{12} )$&   \\
12 & $\Z_2\o A_4$&$A_4$&$\frac {g-3} 6$&$\d\neq 0$ & $( 6^4, 6^4, 2^{12}, \dots , 2^{12})$&$(2^6, 3^4, 3^4)$   \\
13 & $SL_2(3) $ &  & $\frac {g-2} 6$ &$\d\neq 0$ &$(4^6, 3^8, 3^8, 2^{12}, \dots , 2^{12})$ &    \\
14 & $SL_2(3) $ & & $\frac {g-4} 6$ & &$(4^6, 3^8, 6^4, 2^{12}, \dots , 2^{12})$ &   \\
15 & $SL_2(3) $ &  & $\frac {g-6} 6$ &$\d\neq 0$ &$(4^6, 6^4, 6^4, 2^{12}, \dots , 2^{12})$ &   \\
\hline \hline
16 & $\Z_2\o S_4$ &  &$\frac {g+1} {12}$& &$( 3^{16}, 4^{12}, 2^{24}, \dots , 2^{24} )$&   \\
17 & $\Z_2\o S_4$ &  &$\frac {g-3} {12}$& &$( 6^{8}, 4^{12}, 2^{24}, \dots , 2^{24} )$&  \\
18 & $GL_2(3)$ &  &$\frac {g-2} {12}$&    &$( 3^{16}, 8^{6}, 2^{24}, \dots , 2^{24} )$&  \\
19 & $GL_2(3)$&$S_4$&$\frac {g-6}{12}$&       &$( 6^{8}, 8^{6}, 2^{24}, \dots , 2^{24} )$&$(2^{12}, 3^8, 4^6)$  \\
20 & $W_2$ &  &$\frac {g-5} {12}$& &$( 4^{12}, 4^{12}, 3^{16},  2^{24}, \dots , 2^{24} )$&  \\
21 & $W_2$ & &$\frac {g-9} {12}$& &$( 4^{12}, 4^{12}, 6^{8}, 2^{24}, \dots , 2^{24} )$&   \\
22 & $W_3$ & &$\frac {g-8} {12}$ & &$( 4^{12}, 3^{16}, 8^{6}, 2^{24}, \dots , 2^{24} )$&   \\
23 & $W_3$ & &$\frac {g-12} {12}$& &$( 4^{12}, 6^{8}, 8^{6}, 2^{24}, \dots , 2^{24} )$&   \\
\hline \hline
24 & $\Z_2\o A_5$ & & $\frac {g+1} {30}$& &$(3^{40}, 5^{24}, 2^{60}, \dots , 2^{60})$&     \\
25 & $\Z_2\o A_5$ & & $\frac {g-5} {30}$& &$(3^{40}, 10^{12}, 2^{60}, \dots , 2^{60})$&       \\
26 & $\Z_2\o A_5$ & & $\frac {g-15} {30}$& &$(6^{20}, 10^{12}, 2^{60}, \dots , 2^{60})$&    \\
27 & $\Z_2\o A_5$ & $A_5$ & $\frac {g-9} {30}$& &$(6^{20}, 5^{24}, 2^{60}, \dots , 2^{60})$& $( 2^{30}, 3^{20}, 5^{12} )$ \\
28 & $SL_2(5)$& &$\frac {g-14} {30}$& &$(4^{30}, 3^{40}, 5^{24}, 2^{60}, \dots , 2^{60})$&    \\
29 & $SL_2(5)$&  &$\frac {g-20} {30}$& &$(4^{30}, 3^{40}, 10^{12}, 2^{60}, \dots , 2^{60})$&     \\
30 & $SL_2(5)$&  &$\frac {g-24} {30}$ & &$(4^{30}, 6^{20}, 5^{24}, 2^{60}, \dots , 2^{60})$&    \\
31 & $SL_2(5)$&  &$\frac {g-30} {30}$ & &$(4^{30}, 6^{20}, 10^{12}, 2^{60}, \dots , 2^{60})$&    \\
\hline \hline
\end{tabular}
\end{center}
\vspace{3mm} \caption{$\Aut (\X_g)$ and the corresponding signatures}\label{tab_1}
\renewcommand{\arraystretch}{1.24}
\end{table}
\end{tiny}

\begin{remark}
i) In cases 4, 5, and 7-9 we have $n \equiv 0 \mod 2$.
\end{remark}




\subsection{The number of automorphism groups for a fixed genus}
For a fixed $g$ we denote by $N_g$ the number of groups that occur as automorphism groups of genus $g$
curves. We would like to determine what happens to $N_g$ as $g$ increases.

Let $n \in \Z$ such that $n=p_1^{\a_1} \cdots p_s^{\a_s}$. Denote by $\div (n)$ the number of divisors of
$n$. It is well known that $\div (n)= \prod_{i=1}^s (\a_i+1)$. Further, we denote by $\ddiv (n)$ the number
of even divisors of $n$. We have the following lemma:
\begin{lemma}
Let $g$ be fixed. Then the number of automorphism groups that can occur as automorphism groups $\Aut (\X_g)$
of genus $g$ hyperelliptic curves is

i) if $\bAut(\X_g) \iso \Z_n$ then $n_1=\div (g+1)+ \div (2g+1)+\div (2g)-1 $

ii) if   $\bAut(\X_g) \iso D_n$ then $n_2=3\ddiv (g+1) + 2 \ddiv (g) + \div (g) -2 $

iii) if   $\bAut(\X_g) \iso A_4$ and $g > 6$ then $n_3=1$

iv) if  $\bAut(\X_g) \iso S_4$ then $n_4=1$ or 0.

v) if  $\bAut(\X_g) \iso A_5$ then $n_5=1$ or 0.
\end{lemma}

\begin{proof} The proof is elementary and we skip the details.
\end{proof}
%
%
%
%



\section{Moduli spaces of covers}
Fix an integer $g\ge2$ and a finite group $G$. Let $C_1,...,C_r$ be nontrivial conjugacy classes of $G$. Let
$\bC=(C_1,...,C_r)$, viewed as unordered tuple, repetitions are allowed. We allow $r$ to be zero, in which
case $\bC$ is empty. Consider pairs $(\X,\mu)$, where $\X$ is a curve and $\mu: G\to\Aut(\X)$ is an injective
homomorphism. We will suppress $\mu$ and just say $\X$ is a curve with $G$-action, or a $G$-curve. Two
$G$-curves $\X$ and $\X'$ are called equivalent if there is a $G$-equivariant isomorphism $\X\to \X'$.

We say a $G$-curve $\X$ is \textbf{of ramification type} $(g, G, \bC)$ if the following holds: Firstly, $g$
is the genus of $\X$. Secondly, the points of the quotient $\X/G$ that are ramified in the cover $\X\to \X/G$
can be labeled as $p_1,...,p_r$ such that $C_i$ is the conjugacy class in $G$ of distinguished inertia group
generators over $p_i$ (for $i=1,...,r$). (Distinguished inertia group generator means the generator that acts
in the tangent space as multiplication by $\exp(2\pi\sqrt{-1}/e)$, where $e$ is the ramification index). For
short, we will just say $\X$ is of type $(g,G,\bC)$.

If $\X$ is a $G$-curve of type $(g,G,\bC)$ then the genus $g_0$ of $\X/G$ is given by the Riemann-Hurwitz
formula. Define $\H=\H(g,G,\bC)$ to be the set of equivalence classes of $G$-curves of type $(g,G,\bC)$. By
covering space theory, $\H$ is non-empty if and only if $G$ can be generated by elements
$\a_1,\b_1,...,\a_{g_0},\b_{g_0},\g_1,...,\g_r$ with $\g_i\in C_i$ and $\prod_{j}\ [\a_j,\b_j]\  \prod_{i}
\g_i =  1$, where $[\a,\b]=\ \a^{-1}\b^{-1}\a\b$.

Let $\M_g$ be the moduli space of genus $g$ curves, and $\M_{g_0,r}$ the moduli space of genus $g_0$ curves
with $r$ distinct marked points, where we view the marked points as unordered (contrary to usual procedure).
Consider the map $\Phi:\ \H\ \to \ \M_{g}$, forgetting the $G$-action, and the map $\Psi:\ \H\ \to \
\M_{g_0,r} $ mapping (the class of) a $G$-curve $\X$ to the class of the quotient curve $\X/G$ together with
the (unordered) set of branch points $p_1,...,p_r$. If $\H\ne\emptyset$ then $\Psi$ is surjective and has
finite fibers, by covering space theory. Also $\Phi$ has finite fibers, since the automorphism group of a
curve of genus $\ge2$ is finite. By \cite{Be}, the set $\H$ carries a structure of quasi-projective variety
(over $\C$) such that the maps $\Phi$ and $\Psi$ are finite morphisms. If $\H\ne\emptyset$ then all
components of $\H$ map surjectively to $\M_{g_0,r}$ (through a finite map), hence they all have the same
dimension
\[\d(g,G,\bC):= \ \ \dim\ \M_{g_0,r} \ \ = \ \ 3g_0-3+r\]   Since also $\Phi$ is a finite map, we get

\begin{lemma} \label{Lemma1}
Let $\M(g,G,\bC)$ denote the image of $\Phi$, i.e., the locus of genus $g$ curves admitting a $G$-action of
type $(g,G,\bC)$. If this locus is non-empty then each of its components has dimension $\d(g,G,\bC)$.
\end{lemma}

For a description of some of the loci $\M(g,G,\bC)$, not in the hyperelliptic locus, in terms of the theta
nulls see \cite{PS}.

Next we focus on the hyperelliptic locus.  Let $\phi_0: \X_g \to \bP^1$ be the cover which corresponds to the
degree 2 extension $K/k(X)$. Then, $\psi:= \phi \circ \phi_0$ has monodromy group $G:=\Aut(\X_g)$.  From
basic covering theory, the group $G$ is embedded in the group $S_n$, where $n=\deg \psi$. There is an
$r$-tuple $\bar \s := (\s_1, \dots , \s_r)$, where $\s_i\in S_n$ such that  $\s_1, \dots , \s_r$ generate $G$
and $\s_1 \cdots \s_r=1$. The signature of $\psi$ is an $r$-tuple of conjugacy classes  $\bC :=(C_1, \dots ,
C_r)$ in $S_n$ such that $C_i$ is the conjugacy class of $\s_i$.  We use the notation $n^p$ to denote the
conjugacy class of permutations which are a product of $p$ cycles of length $n$. Using the signature of
$\phi: \bP^1 \to \bP^1$  and the Riemann-Hurwitz formula, one finds out the signature of $\psi: \X_g \to
\bP^1$ for any given  $g$ and $G$.

The following theorem describes the list of all the subvarieties of the hyperelliptic moduli which are
determined by group actions.

\begin{thm}
For each  $g\geq 2$, the groups $G$  that occur as automorphism groups and their signatures $\bC$ are given
in Table~\ref{tab_1}. Moreover; the locus  $\H (G, \bC)$ in $\H_g$ of curves with automorphism group $G$ and
signature $\bC$ is an irreducible algebraic variety of dimension $\d (G, \bC)$ as given in Table~\ref{tab_1}.
\end{thm}

\begin{proof}The dimension of each locus is an immediate consequence of Lemma~\ref{Lemma1}. Next, we will
show the irreducibility of the Hurwitz space $\H (G, \bC)$ in each case.

The cases 1-3, 10-15, and 24-31 of Table~\ref{tab_1} follow from the results of \cite{jaa} and \cite{a5}
respectively. It is left to prove cases 4-9 and 16-23 which correspond to the cases when $\G$ is isomorphic
to $D_n$ and $A_5$ respectively. To prove this we make use the GAP and the Braid program written by K.
Magaard. For each case we construct the group $G$ as a subgroup of $S_{|G |}$. In each case we find a
generating tuple and compute its braid action. There is only one braid orbit which shows that the
corresponding space is irreducible. This completes the proof.

\end{proof}

\section{Determining equations of families of curves}
In this section we state the equations of curves in each case of Table~\ref{tab_1}. For a more detailed
treatment of these spaces, including proofs, the reader can check results in  \cite{Bu}, \cite{a5},
\cite{jaa}, \cite{issac}, among others. Recall that  $\G$ is the monodromy group of a cover $\ff: \bP^1 \to
\bP^1$ with signature $\bar \s:=(\s_1, \s_2, \s_3)$ as in  Table~\ref{tab_1}.  We fix the coordinates in
$\bP^1$ as $x$ and $z$ respectively and from now on denote the cover $\ff : \bP^1_x \to \bP^1_z$.  Thus, $z$
is a rational function in $x$ of degree $|\G|$. We denote by $q_1, q_2, q_3 $ the corresponding branch points
of $\ff$. Let $S$ be the set of branch points of $\f: \X_g \to \bP^1$. Clearly $q_1, q_2, q_3 \in S$. As
above $W$ denotes the images in $\bP^1_x$ of Weierstrass points of $\X_g$ and $V:=\cup_{i=1}^3 \ff^{-1}
(q_i)$.

Let $$z= \frac {\nf (x) } { \df (x) }$$ where $\nf, \df \in k[x]$.  For each branch point $q_i$, $i=1,2,3$ we
have $z-q_i = \frac {p(x)} {\df (x)}$. Hence, we have the degree $|\G|$ equation
$$ p(x)=\nf (x)  - q_i \cdot \df (x) =0,$$
where the multiplicity of all the roots is the same and correspond to the ramification index of $q_i$ (i.e.,
the index of the normalizer in $\G$ of $\s_i$). We denote the ramification of $\ff: \bP^1_x \to \bP^1_z$, by
\begin{equation}
\begin{split}
\varphi_m^r (x) & : = \nf (x)  - q_1 \cdot \df (x),\\
\chi_n^s  (x)    & :=  \nf (x)  - q_2 \cdot \df (x) \\
  \psi_p^t (x) & :=   \nf (x)  - q_3 \cdot \df (x),  \\
  \end{split}
\end{equation}
where the subscript denotes the degree of the polynomial and the superscript is the index of the normalizer
in $\G$ of $\s_i$.

It is obvious that
\[\phi^{-1} \left( S \setminus \{ q_1, q_2, q_3\} \right) \subset W. \]
Let  $\l \in S \setminus \{ q_1, q_2, q_3\}$. The points in the fiber $\phi^{-1} (\l)$ are the roots of the
equation:
\begin{equation}\label{eq_l}
 \nf (x) -\l \cdot \df (x) =0
\end{equation}
To determine the equation of the curve we simply need to determine the Weierstrass points of the curve. Let
\[G(x) := \prod_{\l \in S \setminus \{ q_1, q_2, q_3\}} \left( \nf (x) -\l \cdot \df (x)    \right) .\]

For each fixed $\phi$ there are the following cases and the corresponding equation of the curve is $y^2=f(x)$
where $f(x)$ is as below:

\begin{equation}
\begin{split}
\aligned
1)\quad & V \cap W = \emptyset,  \\
2)\quad & V \cap W =\f^{-1} (q_1),  \\
3)\quad & V \cap W =\f^{-1} (q_2),  \\
4)\quad & V \cap W =\f^{-1} (q_3),  \\
5)\quad & V \cap W =\f^{-1} (q_1) \cup \f^{-1} (q_2),  \\
6)\quad & V \cap W =\f^{-1} (q_2) \cup \f^{-1} (q_3),  \\
7)\quad & V \cap W =\f^{-1} (q_1) \cup \f^{-1} (q_3),  \\
8)\quad & V \cap W =\f^{-1} (q_1) \cup \f^{-1} (q_2) \cup \f^{-1} (q_3), \\
\endaligned
\qquad \aligned
&  G(x), \\
& \varphi (x) \cdot G(x) ,\\
& \chi (x)\cdot G(x) ,\\
&  \psi(x)\cdot G(x),\\
&  \varphi(x)\cdot \chi(x)\cdot G(x), \\
&  \chi(x)\cdot \psi(x)\cdot G(x),\\
&  \varphi(x)\cdot  \psi(x)\cdot G(x),\\
& \varphi(x)\cdot  \chi(x)\cdot  \psi(x)\cdot G(x)\\
\endaligned
\end{split}
\end{equation}

Since we know $z = \frac {\nf (x) } { \df (x) }$ in each case, then we can easily compute the equation of the
curve for all cases of Table~\ref{tab_1}.

\begin{remark}
When $\G = \Z_n, D_n, A_4$, then two of the branch points of $\ff : \bP^1 \to \bP^1$ correspond to the same
conjugacy class. Then, cases 2) and 3) are the same. So are also the cases 6) and 7). This explains the
number of cases in Table~\ref{tab_1}.
\end{remark}

\begin{remark}
The polynomial $G(x)$ can also be computed by computing an orbit of $\G$ using the generators in
Eq.~\eqref{eq1}. This follows from the fact that $\G$ is the monodromy group of $\phi: \bP^1_x \to \bP^1_z$
and therefore has a complete orbit on the fiber $\phi^{-1} (\l)$ for each $\l \neq q_1, q_2, q_3$.
\end{remark}

\subsection{$\G \iso \Z_n$} The branch points of the cover $\phi (x) = x^n$ are $0$ and $ \infty$.
For $\l \in S \setminus \{ 0, \infty\}$, the points $\phi^{-1} (\l)$ satisfy the equation $G_\l (x)= x^n-\l$.

We have
\begin{equation}
G(x)= \prod_{\l \in S \setminus \{ 0, \infty\}} G_{\l} (x)=
 x^{nt}+ \dots + a_i x^{n(t-i)} + \dots a_{t-1} x^n+1,
\end{equation}
Then $\X_g$ belongs to cases 1, 2, 3 in Table~\ref{tab_1}. The equation of each family is $y^2=F(x)$,  where
$F(x)$ is $ G(x), G(x), x\cdot  G(x)$ and  $t$ respectively $ \frac {2g+2} n, \frac {2g+1} n, \frac {2g} n$.
See \cite{jaa} for details.
%
\subsection{$\G \iso D_n$}
In this case, the branch points of $z(x)$ are $\infty$, and $\pm 2$. We have $G(x)$ as below:
$$G(X)= \prod_{i=1}^\d (X^{2n} - \l_i X^n +1).$$
Then,
\begin{equation}
\begin{split}
G(X)= & X^{2nt}+ a_1 X^{2nt-n} + \dots + a_t X^{nt} + a_{t-1} X^{(n-1)t} + \dots  + a_1 X^n +1
\end{split}
\end{equation}
where $a_i, i=1,\dots t$ are polynomials in terms of the symmetric polynomials $s_1, \dots , s_t$ of $\l_i$
(i.e., etc.).
\[a_1=s_1, a_2=t+s_2, a_3=(t-1)s_1+s_3, a_4:=\left( \overset{t} {n/2}  \right) + (t-2)s_2 +s_4, \]
Then, each family is parameterized as in Table~\ref{tab2}.

\subsection{$\G \iso  A_4$} The branch points of the cover $\phi$ are $\{ \infty, 6 i \sqrt{3}, - 6i\sqrt{3}\}$.
The polynomials over these branch points are
\begin{equation}
\begin{split}
\varphi_m (x) & : = x^4+2i\sqrt{3} x^2 +1,\\
\chi_n  (x)    & :=  x^8+14x^4+1,\\
  \psi_p (x) & :=    x(x^4-1) \\
  \end{split}
\end{equation}
For $\l \in S \setminus \{\infty, 6 i \sqrt{3}, - 6i\sqrt{3} \}$ (equivalently $\l^2+108\neq 0$) we have
%
\begin{equation}  \label{G}
G_{\l} (x)=x^{12} - \l x^{10} - 33 x^8 + 2 \l x^6 - 33 x^4 - \l x^2+1,
\end{equation}
%
There are $\d = \frac {g+1} 6$ points in $S \setminus \{ \infty, \pm 6 i \sqrt{3}\}$. Denote by
\[G(x):=\prod_{i=1}^{\delta }\left(x^{12}- \l_i x^{10}-33 x^8 + 2 \l_i x^6 -33 x^4-\l_i x^2+1\right)\]
Then, each family is parameterized as cases 10-15 in Table~\ref{tab2}.
%
\subsection{$\G \iso S_4$ }
The branch points of $\phi (x)$ are $\{ 0, 1, \infty\}$. Then,
\begin{equation}
\begin{split}
\varphi(x):= & x^{12} - 33 x^8 - 33 x^4 +1, \\
 \chi(x):= & x^8 + 14 x^4 + 1, \\
  \psi(x):= & x^4-1.
\end{split}
\end{equation}
For $\l \in S \setminus \{ 0, 1, \infty\}$, points in $\phi^{-1} (\l)$ are roots of the polynomial
\[ G_\l (x)=x^{24}+ \l x^{20}+(759-4\l)x^{16}+2(3\l+1288)x^{12}+(759-4\l)x^{8}+\l x^{4}+1\]
There are $\d $ points in $S \setminus \{ 0, 1, \infty\}$, where $\d$ is given as in Table~\ref{tab_1}; see
\cite{jaa} for details. We denote
\begin{equation}\label{eqM}
M(x):=   \prod_{i=1}^\d  G_{\l_i} (x)
\end{equation}
%
\subsection{$\G \iso  A_5$}
The branch points of $\phi $ are 0, 1728 and $\infty$.  At the place $z=1728$ the function has the following
ramification:
\[\phi (x) - 1728= - \frac {\left( x^{30}+522x^{25}-10005x^{20}-10005x^{10}-522x^5+1\right)^2 }
{ x^5 \, \left(x^{10} + 11x^5-1\right)^5}.\]
Then,
\begin{equation}
\begin{split}
\varphi (x) & = x^{20}-228 x^{15}+494x^{10}+228x^5+1 \\
\chi (x) & = x \ (x^{10} + 11x^5 -1)\\
\psi (x) & =  x^{30}+522x^{25}-10005x^{20}-10005x^{10}-522x^5+1.\\
\end{split}
\end{equation}
Then for each $\l_i \in S \setminus \{0, 1728, \infty\}$ the places in $\phi^{-1} (\l_i)$ are the roots of
the following polynomial
\begin{small}
\begin{equation*}
\begin{split}\scriptstyle
& G_i (x) = -x^{60}+(684 - \l_i) x^{55}-(55 \l_i + 157434) x^{50}-(1205 \l_i-12527460 ) x^{45} \\
& -(13090\l_i\!+\!77460495)x^{40}\!+\!(130689144\!-\!69585\l_i) x^{35}\!+\!(33211924\!-\!134761\l_i)x^{30} \\
& +(69585\l_i-130689144) x^{25}\!-(13090 \l_i\!+\!77460495)x^{20}\!-(12527460-1205\l_i) x^{15} \\
& -(157434+ 55 \l_i) x^{10}+( \l_i-684) x^5-1
\end{split}
\end{equation*}
\end{small}
Then, we let
\begin{equation}\label{eqL}
\Lambda (x) := \prod_{i=1}^\d G_i (x)
\end{equation}



\begin{table}[ht!]
\begin{center}
\begin{tabular}{||c|c||}
\hline \hline
\#  & $y^2=f(x) $  \\
 \hline \hline
 1 &    $x^{2g+2} + a_1 x^{n (t-1)} + \dots + a_\d x^n +1$, \ \   $ t= \frac { 2g +2} n$  \\
 2 &  $x^{2g+1} + a_1 x^{n (t-1)} + \dots + a_\d x^n +1$, \ \   $ t=\frac { 2g + 1} n$ \\
 3 & $x ( x^{nt} + a_1 x^{n (t-1)}   + \dots + a_\d x^n +1 )$, \ \  $ t= \frac {2g} n $  \\
   \hline \hline
 4 &    $F(x):= \prod_{i=1}^t (x^{2n} + \l_i x^n +1)$, \ \  $t =\frac {g+1} n$\\
 5 &    $(x^n-1) \cdot F(x)$           \\
 6 &     $x \cdot F(x)$             \\
 7 &      $(x^{2n}-1) \cdot F(x)$             \\
 8 &        $x(x^n-1) \cdot F(x)$             \\
 9 &       $x(x^{2n}-1) \cdot F(x)$              \\
\hline \hline 10 &  $G(x):=\prod_{i=1}^{\delta }
\left(x^{12} - \l_i x^{10} - 33 x^8 + 2 \l_i x^6 - 33 x^4 - \l_i x^2+1 \right)$   \\
11 &     $(x^4+2i \sqrt{3}x^2 +1)\cdot  G(x)$           \\
12 &      $ (x^8+14 x^4+1)  \cdot G(x)$         \\
13 &     $x (x^4-1)  \cdot G(x) $        \\
14 &     $x (x^4-1)(x^4+2i \sqrt{3} x^2 +1) \cdot G(x)$        \\
15 &     $x(x^4-1) (x^8+14x^4+1)\cdot G(x)$           \\
\hline \hline
16 &     $M(x)$           \\
17 &    $S(x) \cdot M(x)$          \\
18 &     $T(x) \cdot M(x)$           \\
19 &     $S(x) \cdot T(x) \cdot M(x)$           \\
20 &    $R(x) \cdot M(x)$           \\
21 &     $R(x) \cdot S(x) \cdot M(x)$          \\
22 &   $R(x) \cdot T(x) \cdot M(x)$          \\
23 &     $R(x) \cdot S(x) \cdot T(x) \cdot M(x)$          \\
\hline \hline
24 &        $\Lambda (x)$        \\
25 &      $\left(x^{20}-228 x^{15}+494x^{10}+228x^5+1\right) \cdot \Lambda (x) $         \\
26 &      $\left( x \ (x^{10} + 11x^5 -1) \right) \cdot \Lambda (x)$         \\
27 &      $ \psi\cdot \Lambda (x)$          \\
28 &     $ \left(x^{20}-228 x^{15}+494x^{10}+228x^5+1\right) \cdot \left( x \ (x^{10} + 11x^5 -1) \right)
\cdot \Lambda (x)$        \\
29 &       $ \left( x \ (x^{10} + 11x^5 -1)\right)\cdot \psi \cdot \Lambda (x)$      \\
30 &        $\left(x^{20}-228 x^{15}+494x^{10}+228x^5+1 \right) \cdot \psi\cdot \Lambda (x)$     \\
31 &$\left(x^{20}-228 x^{15}+494x^{10}+228x^5+1   \right) \cdot \left( x \ (x^{10} + 11x^5 -1)\right) \cdot
\psi\cdot \Lambda (x)$   \\
\hline \hline
\end{tabular}
\end{center}
\vspace{3mm} \caption{The equation of the generic hyperelliptic curve with group $G$.}\label{tab2}
\renewcommand{\arraystretch}{1.24}
\end{table}
\noindent where $M(x)$ is as in Eq.~\eqref{eqM} and $\Lambda (x)$ as in Eq.~\eqref{eqL}.

Further, we notice that curves with automorphism do have something in common.  Let $I_4$ be the invariant of
binary forms defined in terms of transvections as in \cite{issac}.
\begin{lemma}
Let $\X_g$ be a hyperelliptic curve given with equation $y^2=f(x)$ such that $|\Aut (\X_g)|> 2$. Then, $ I_4
(f) =0$.
\end{lemma}

\section{Lattice of groups and inclusion among the loci}
In this section we will study the inclusions between the automorphism groups for a fixed genus $g$.  This was
also studied in \cite{We}, however, the author focuses only on automorphism groups with  reduced automorphism
groups isomorphic to a dihedral group.

For any genus $g$ we determine completely the lattice of loci $\H_g^G$ in $\H_g$.  Using GAP we can determine
the list of all groups that occur as automorphism groups of genus $g$. This can be done for any fixed genus
$g\leq 299$.

We identify groups in GAP by their identity in the library of SmallGroups.  This library contains only groups
of order up to 2400. We know that the order of the automorphism group is $\leq 8(g+1)$. Hence, we can
determine the list of groups for all $g \leq 299$.

Let $L$ be the list of all groups occurring as automorphism groups of genus $g$ hyperelliptic curves. Each
entry in $L$ is an ordered pair $[m, n]$ where $m$ denotes the order of the group and $n$ the position that
this group is stored in the Gap library. We order $L$ as follows
\[ [m, n] < [r,s] \quad \textit{ if } \ \ m \leq r \textit{ and } n \leq s\]
Consider $L= \{ G_1, \dots G_N\} $ such that
\[ G_1 < G_2 < \dots < G_N\]
with respect to the above ordering. The incidence matrix of $L$ is
\[ M=[m_{i,j}]\]
where
\[m_{i, j} =\left\{
\aligned & 1, \textit{   if } G_i \textit{ is a subgroup of  } G_j\\
& 0, \textit{otherwise}
\endaligned
\right.
\]

Then $M$ is an upper triangular $N\times N$ matrix. Such matrix can be easily determined for any $g$. We have
implemented a program in GAP which gives the lattice and the incidence matrix for any $g\geq 2$.

\begin{ex}
The lattice of the groups for genus 4 is given in Fig.~\ref{gen4}. Each group is presented by its {\sc GAP}
identity. Each level contains cases with the same dimension.  \\

\begin{tiny}
\begin{center}
\begin{figure}[ht!]
\xymatrix{
\delta= 7  &            &            &(2,1), \Z_2 \ar@{-}[dr] \ar@{-}[dd] \ar@{-}[dddll]  &     \\
\delta= 4  &            &            &                                  &(4,2),  V_4 \ar@{-}[ddd] \ar@{-}[ddl]&  \\
\delta= 3  &            &            &(4, 1), \Z_4\ar@{-}[d]\ar@{-}[ddl]                &  \\
\delta= 2  &(6, 2), \Z_6 \ar@{-}[dd]\ar@{-}[ddr]   &   &(8, 3), D_8\ar@{-}[d]      &   \\
\delta= 1  &            &(8, 4),  G_2 \ar@{-}[d]\ar@{-}[dr]    &(16, 7), D_{16} \ar@{-}[d] &(20, 4),   \Z_2 \times D_{10} \ar@{-}[d] \\
\delta= 0  &(18, 2), \Z_{18}  &(24, 3), SL_2 (3) &(32, 19), U_8                            &(40, 8), V_{10}  \\
}
\caption{The lattice of automorphism groups for hyperelliptic curves of  $g=4$.}
\end{figure}
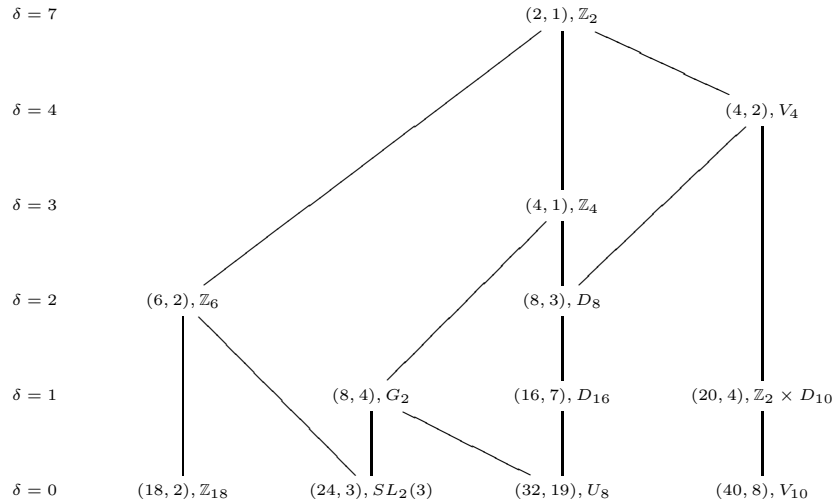\label{gen4}
\end{center}
\end{tiny}
\end{ex}

\section{Concluding remarks}

This main goal of this paper was to give a more unified algebraic approach of the case of automorphism groups
of algebraic curves in zero characteristic. As part II of this project we intend to ask the same questions on
characteristic $p > 0$. Such questions are very much unexplored for algebraic curves defined over fields of
positive characteristic.

One would like to describe such  loci  in terms of invariants of the curves as already done for genus $g=2,
3$; see \cite{GS, GSS}. There have been attempts to do this in \cite{jaa, issac, a5} using invariant of
binary forms. Since such invariants are unknown for degree $\geq 7$ which makes this method unlikely to
succeed. Further, invariants of binary forms are huge polynomials in terms of the coefficients of the curve.
Even, if they are completely known they would be computationally not efficient. Hence, one is tempted to try
to describe such loci in terms of theta nulls; see \cite{PS, PS2, PSW}.

\end{document}